\documentclass[11pt]{amsart} 
\usepackage{hyperref}
\usepackage{amsmath, amsthm,url,
  amssymb,setspace,epsfig, mathrsfs,fontenc, 
  fullpage, MnSymbol}

\newtheorem{lem}{Lemma}

\newtheorem{prop}{Proposition}
\newtheorem{thm}{Theorem}

\newtheorem{cor}{Corollary}

\newcommand{\nc}{\newcommand}
\nc{\renc}{\renewcommand}
\nc{\ssec}{\subsection}
\nc{\sssec}{\subsubsection}
\nc{\on}{\operatorname}
\nc {\ra} {\rightarrow}
\nc {\inj} {\hookrightarrow}
\nc {\lra} {\longleftrightarrow} 
\nc         {\rar}[1]       {\stackrel{#1}{\longrightarrow}}

\nc {\Sl}{\mathfrak{sl}}
\nc{\fg}{\mathfrak{g}}
\nc{\fI}{\mathfrak{i}}
\nc{\fh}{\mathfrak{h}}
\nc{\fb}{\mathfrak{b}}
\nc{\fu}{\mathfrak{u}}
\nc{\fn}{\mathfrak{n}}
\nc{\hfg}{\widehat{\fg}}
\nc{\hfh}{\widehat{\fh}}

\nc {\hH}{{\check{H}}}
\nc {\hB}{{\check{B}}}
\nc {\hN}{{\check{N}}}
\nc {\hG}{{{\check{G}}}}
\nc {\cfg}{\check{\fg}}
\nc {\cfb}{\check{\fb}}
\nc {\cfn}{\check{\fn}}
\nc {\cfh}{\check{\fh}}
\nc {\clambda}{\check{\lambda}}

\nc {\bone}{\mathbf{1}}
\nc {\bu}{\mathbf{u}}
\nc {\bv}{\mathbf{v}}
\nc {\bw}{\mathbf{w}}

\nc{\bC}{\mathbb{C}}
\nc{\bM} {\mathbb{M}}
\nc{\bU} {\mathbb{U}}
\nc{\bW}{\mathbb{W}}
\nc{\bL}{\mathbb{L}}
\nc {\bZ} {\mathbb{Z}}
\nc {\bGm} {\mathbb{G}_m}

\nc {\cD}{\mathcal{D}}
\nc{\cZ}{\mathcal{Z}}
\nc {\cDt}{\cD^\times} 
\nc {\cA} {\mathcal{A}}
\nc {\cI}{\mathcal{I}}
\nc {\cR}{\mathcal{R}}
\nc {\cJ}{\mathcal{J}}
\nc {\cS}{\mathcal{S}}
\nc {\cO}{\mathcal{O}}

\nc {\tU} {\widetilde{U}}

\nc {\Ind}{\mathrm{Ind}}
\nc {\Hom}{\mathrm{Hom}}
\nc {\crit}{\mathrm{crit}}
\nc {\Res}{\mathrm{Res}} 
\nc {\ev}{\mathrm{ev}}
\nc{\Spec}{\mathrm{Spec}\,}
\nc {\ord}{\mathrm{ord}}
\nc {\Op}{\mathrm{Op}}
\nc {\MOp}{\mathrm{MOp}}
\nc {\Conn}{\mathrm{Conn}}
\nc {\MT}{\mathrm{MT}}
\nc{\Sym}{\mathrm{Sym}}
\nc {\modd}{\mathrm{mod}}
\nc {\ad}{\mathrm{ad}}
\nc {\Tr}{\mathrm{Tr}}
\nc{\Kil}{\mathrm{Kil}}
\nc {\End}{\mathrm{End}}
\nc {\can}{\mathrm{can}}
\nc {\Aut}{\mathrm{Aut}}
\nc {\Der}{\mathrm{Der}}
\nc {\gr}{\mathrm{gr}}
\nc {\res}{\mathrm{res}}

\nc {\bP} {\bar{P}}
\nc {\bJ}{\bar{J}}

\usepackage[all]{xy}

\nc {\llp} {\mathopen{ (\!(}}
\nc {\rrp} {\mathopen{ )\!)}}
\nc {\llb} {\mathopen{ [\![}}
\nc {\rrb} {\mathopen{ ]\!]}}
\nc {\lc} {\mathopen{:\!}}
\nc {\rc}{\mathopen{\!:}}
\nc{\fgbb} {\fg\llb t \rrb}
\nc{\fgpp} {\fg\llp t \rrp}
\nc{\fhbb} {\fh\llb t \rrb}
\nc{\fhpp} {\fh\llp t \rrp}
\nc{\bCpp} {\bC \llp t \rrp}
\nc{\bCbb} {\bC \llb t \rrb}
\nc{\hNpp} {\hN \llp t \rrp}
\nc{\hNbb} {\hN \llb t \rrb}
\nc{\cfbpp} {\cfb \llp t \rrp}
\nc{\cfbbb} {\cfb \llb t \rrb}

\begin{document}
\title{Compatibility of the Feigin-Frenkel Isomorphism and the Harish-Chandra Isomorphism for jet algebras}  
\author{Masoud Kamgarpour} 
\email{masoud@uq.edu.au}
\address{School of Mathematics and Physics, The University of Queensland}
\date{\today}

\thanks{The problems addressed in this paper arose in joint work with Travis Schedler \cite{MasoudTravis1} \cite{MasoudTravis2}. I would like to thank him for many helpful conversations about this paper. I wish to thank M. Duflo, E. Frenkel, D. Gaitsgory, A. Molev, S. Raskin,  R. Zhang, and X. Zhu for their encouragement and helpful conversations. I  would also like to thank the Max Planck Institute for Mathematics in Bonn for its hospitality.}



\begin{abstract} Let $\fg$ be a simple finite-dimensional complex Lie algebra with a Cartan subalgebra $\fh$ and Weyl group $W$. Let
 $\fg_n$ denote the Lie algebra of $n$-jets on $\fg$. A theorem of Rais and Tauvel and Geoffriau identifies the centre of the category of $\fg_n$-modules with the algebra of functions on the variety of $n$-jets on the affine space $\fh^*/W$. On the other hand, a theorem of Feigin and Frenkel identifies the centre of the category of critical level smooth modules of the corresponding affine Kac-Moody algebra with the algebra of functions on the ind-scheme of opers for the Langlands dual group. We prove that these two isomorphisms are  compatible by defining the higher residue of opers with irregular singularities. We also define generalized Verma and Wakimoto modules and relate them by a nontrivial morphism. 
\end{abstract}

\maketitle

\tableofcontents


\section{Introduction} 
 
Representation theory of affine Kac-Moody algebra at the critical level plays a crucial role in Beilinson and Drinfeld's approach to the geometric Langlands program \cite{BD}. An essential ingredient of this approach is a correspondence, due to Feigin and Frenkel \cite{FF}, between the center of the (completed) universal enveloping algebra of an affine Kac-Moody algebra at the critical level and the geometry of certain operators, known as \emph{opers}, associated with the Langalnds dual datum. Building on \cite{BD}, Frenkel and Gaitsgory \cite{BIBLE} formulated the local geometric Langlands correspondence. In a series of subsequent papers, the authors established several cases of this program, mostly concerning the unramified and tamely ramified situations. Opers are meromorphic connections equipped with extra structure, and the unramified (resp. tamely ramified) local geometric Langlands program concerns those opers whose underlying connections are holomorphic (resp. meromorphic with singularities of order at most one).

The full, and still very conjectural, local geometric Langlands program \cite{BIBLE} requires studying also the \emph{irregular opers}, i.e.,  opers whose underlying connections have singularities bigger than one. The latter opers are related to the wildly ramified part of the local geometric Langlands program, since it is known that there is a deep and mysterious analogy between wildly ramified representations of the Galois group and irregular singularities; see, for instance, \cite{Katz}. There are, currently, few publications dealing with the wildly ramified part of the local geometric Langlands program; see, for instance, \cite{Laredo} and \cite{Fedorov}. The present paper is a contribution in this area. 

In more details, we determine the central support of certain critical level representations by defining the notion of higher residue of irregular opers. Furthermore, we define two classes of these modules: \emph{generalized Verma} and \emph{generalized Wakimoto modules}. In contrast to the usual Verma and Wakimoto modules,  the representations we consider are not in category $\cO$, nor are they Iwahori integrable; in particular, we do not know how to define the character of these representations.  Nevertheless, we expect that they have some of the properties of the usual Verma and Wakimoto modules. For example, we prove that there exists a nontrivial morphism from the generalized Verma to the generalized Wakimoto module. We expect that, under appropriate conditions, this morphism is an isomorphism. 
Our constructions and results should be related to the categorification of ramified principal series representations in the setting of local geometric Langlands program, since the generalized Verma modules are analogues of families of principal series representations studied in \cite{MasoudTravis1} and \cite{MasoudTravis2}.

\ssec{Recollections on the centre of enveloping algebras}\label{ss:recollections}

\sssec{Harish-Chandra's Isomorphism}  
Let $\fg$ be a simple finite dimensional Lie algebra of rank $\ell$ over the complex numbers and let $\cZ(\fg)$ denote the centre of the universal enveloping algebra of $\fg$. Note that $\cZ(\fg)$ is isomorphic to the (Bernstein) centre of the category of representations of $\fg$. Let $\fh$ be a Cartan subalgebra and let $W$ denote the  Weyl group. Let $\fh^*=\Hom(\fh,\bC)$ denote the dual vector space to $\fh$. Harish-Chandra's Isomorphism states that the commutative algebra $\cZ(\fg)$ is canonically isomorphic to $W$-invariant functions on $\fh^*$. Since $W$ is a finite group acting on an affine variety $\fh^*$, one can define the structure of an affine variety  on $\fh^*/W$; thus, we have $\cZ(\fg) \simeq \bC[\fh^*/W]$. 

\sssec{Jets and loops} Given a ``space" $X$, one can consider the space of jets and loops on $X$. In algebraic geometry, the jet space $X\llb t \rrb$ (resp. loop space $X\llp t \rrp$) is defined by its functor of points: 
\[
X\llb t \rrb: = X(R\otimes \bC\llb t \rrb),\quad \quad (\textrm{resp.}\, X\llp t \rrp(R) := X(R\otimes \bC\llp t \rrp)),
\]
 where $R$ is a $\bC$ algebra. Similarly, we can define the $(n-1)$-jets on $X$ by replacing $R\otimes \bC\llb t \rrb$  with $R\otimes \bC\llb t \rrb/t^{n}$. 
 In general, if $X$ is a scheme, then $X\llb t \rrb$ and $X_n$ have the structure of a scheme. However, $X\llp t \rrp$ may not have a reasonable algebra-geometric realization; this happens, for example, when $X$ is the flag variety.\footnote{This issue is the source of a lot of  problems in  geometric representation theory and the geometric Langlands program, including the fact that dealing with Wakimoto modules is rather difficult.} If $X$ is affine, however, one knows that $X\llp t \rrp$ has a canonical structure of a (strict) ind-scheme. 
 
 \sssec{Harish-Chandra's Isomorphism for Jet algebras} Thus, we have the schemes $\fg_n$ and $(\fh^*/W)_n$. It is easy to see that $\fg_n$ is endowed with a canonical Lie bracket coming from $\fg$; in fact, $\fg_n\simeq \fg\otimes \bC[t]/t^n$.   The algebras $\fg_n$ are known as \emph{generalized Takiff algebras} or \emph{truncated current algebras}. Takiff \cite{Takiff} proved that the algebra of invariant polynomials $\bC[\fg_2]^{\fg_2}$ is isomorphic to $\bC[(\fh^*/W)_2]$. Rais and Tauvel \cite{RaisTauvel} generalized this result to arbitrary $n$. Combining their theorem with Duflo's Isomorphism \cite{Duflo}, one obtains the following result. 
 
  \begin{thm} \label{t:jets} 
 There exists a canonical isomorphism of commutative algebras $\cZ(\fg_n) \simeq \bC[(\fh^*/W)_n]$. 
 \end{thm} 
 
Geoffriau \cite{GeoffriauCenter}  proved the above result by generalizing Harish-Chandra's approach to the jet setting; in particular, he was able to determine the action of $\cZ(\fg_n)$ on the Verma modules for $\fg_n$, see \S \ref{sss:Verma}.

\sssec{Loop algebras} We now consider the ind-schemes $\fgpp$ and $(\fh^*/W)\llp t \rrp$. It is easy to see that $\fgpp$ is a Lie algebra; in fact, $\fgpp\simeq \fg\otimes \bCpp$.  It is natural to wonder if the centre of the universal enveloping algebra of $\fgpp$ is related to the algebra of functions on the ind-scheme $(\fh^*/W)\llp t \rrp$. It turns out to be fruitful to consider a more general problem. 
Let $\kappa$ be an invariant bilinear form on $\fg$. The affine Kac-Moody algebra $\hfg_\kappa$ at level $\kappa$ is defined to be the central extension
\[
0\ra \bC.\bone \ra \hfg_\kappa\ra \fgpp \ra 0,
\]
with the two-cocycle defined by the formula 
\[
(x\otimes f(t), y\otimes g(t)) \mapsto -\kappa(x,y).\,\Res_{t=0} fdg.
\]
A module $V$ over $\hfg_\kappa$ is \emph{smooth} if for every $v\in V$ there exits $N_v\geq 0$ such that $t^{N_v} \fgbb$ annihilates $v$, and such that $\bone\in \bC\subset \hfg_\kappa$ acts on $V$ as the identity.  Let $\hfg_\kappa-\modd$ denote the category of smooth $\hfg_\kappa$-modules. Let $\cZ(\hfg_\kappa)$ denote the centre of this category. 
Note that $\hfg_\kappa-\modd$ is equivalent to a category of modules over a certain completion of the universal enveloping of $\hfg_\kappa$ and $\cZ(\hfg_\kappa)$ is isomorphic to the centre of this completed algebra.

\sssec{Feigin and Frenkel's Theorem: coordinate dependent version} The level corresponding to the invariant form $-\frac{1}{2}$ times the Killing form is known as the \emph{critical level} and is denoted by $\kappa_{\crit}$.  Feigin and Frenkel's Theorem states that $\cZ_\kappa(\fg)$ is trivial, except when $\kappa=\kappa_\crit$, in which case it is isomorphic to the (topological) algebra of functions on $(\fh^*/W)\llp t \rrp$. In other words, while the center of $\fgpp-\modd$ is trivial, the centre of $\hfg_\crit-\modd$ has a description similar to Harish-Chandra's description of the centre of $\fg-\modd$.

\sssec{Coordinate independence} 
 Let $X$ be a smooth projective curve over the complex numbers (equivalently, $X$ is a Riemann surface). For applications to the geometric Langlands program, it is important to note that given $\fg$, one can associate to every $x\in X$, an affine Kac-Moody algebra $\hfg_{\kappa,x}$. In particular, for every $x$, we have a topological commutative algebra $\cZ(\hfg_{\kappa_\crit,x})$. If we choose a local coordinate $t$ at $x$, then we obtain an isomorphism 
 \[
\cZ(\hfg_{\kappa_\crit,x}) \simeq \bC[(\fh^*/W)\llp t \rrp].
\]
This isomorphism is not, however, canonical, since the elements of $\cZ(\hfg_{\kappa_\crit,x})$ do not transform as functions on $(\fh^*/W)\llp t \rrp$ under the group of change of coordinates. Rather, they transform as opers. We refer the reader to \cite{FrenkelBook} for an extensive discussion regarding coordinate independence. To recall the definition of opers, we need some notation.

\sssec{Notation} \label{sss:notation} 
Choose a Borel subalgebra $\fb$ containing $\fh$, and let $\fn=[\fb,\fb]$. Let $\cfg$ denote the Langlands dual Lie algebra to $\fg$. Henceforth, we identify $\fh^*$ with a Cartan subalgebra $\cfh$ of $\cfg$. Let $\cfb$ and $\cfn$ denote the corresponding subalgebras of $\cfg$. 
Choose generators $f_i$ for the root subgroups corresponding to the negative simple roots of $\cfg$. Let 
\[
\displaystyle p_{-1}=\sum_{i=1}^{\ell} f_i.
\]
 Let $\hG$ denote the adjoint simple group whose Lie algebra is isomorphic to $\cfg$. Let $\hB$, $\hN$, and $\hH$ denote the subgroups of $\hG$ corresponding to $\cfb$, $\cfn$, and $\cfh$, respectively. 

\sssec{Opers} The notion of opers is due to Beilinson and Drinfeld \cite{opers}. For us, the following description, given in terms of a local coordinate, suffices: a $\hG$-oper (on the punctured disk $\cDt:=\Spec(\bCpp)$) is an $\hNpp$-conjugacy class (also known as a gauge equivalence class) of operators of the form 
\begin{equation} \label{eq:operDt}
\nabla =\partial_t + p_{-1} + x(t),\quad \quad x(t) \in \cfbpp
\end{equation} 
 One knows that $\hG$-opers form an ind-scheme; see, for instance, \cite[Corollary 1.3.2]{BIBLE}. In fact, we have an isomorphism $\Op_{\hG} \simeq (\fh^*/W)\llp t \rrp$, but this isomorphism depends on a choice of coordinate. We let $\Op_{\hG}$ denote the ind-scheme of opers over $\bC$. 
 
 \sssec{Feigin and Frenkel's Theorem: coordinate independent version} 
 
 \begin{thm}[\cite{FF}, \cite{FrenkelBook}] There exists a canonical isomorphism of  commutative topological algebras 
$\cZ(\hfg_\crit) \simeq \bC[\Op_{\hG}]$. \label{t:FF}
\end{thm}

\ssec{Compatibility theorem} \label{ss:compatibility} Our goal is to show that the isomorphisms of Theorems \ref{t:jets} and \ref{t:FF} are naturally compatible. 
 Let us first note that there exists a functor  $\fg_n-\modd\ra \hfg_\kappa-\modd$  defined by  
\begin{equation}\label{eq:bM}
V\mapsto \bM_\kappa(V):= \Ind_{\fgbb\oplus \bC.\bone}^{\hfg_\kappa} (V)
\end{equation} 
where $\fgbb$ acts on $M$ via the canonical map $\ev_n: \fgbb\ra \fg_n=\fgbb/t^n$ and $\bone$ acts as identity. It is natural to ask how the centre $\cZ(\hfg_\crit)$ acts on the module $\bM_\crit(V)$. The first result in this direction was obtained by Beilinson and Drinfeld. We now recall their theorem.

\sssec{} Following \cite[\S 3.8.8]{BD}, we define a $\hG$-oper (on the disk $\cD=\Spec \bC\llb t\rrb$) with singularities of order at most $n$ to be a $\hNbb$-conjugacy class of operators of the form 
\begin{equation}\label{eq:oper}
\nabla =\partial_t +  t^{-n}(p_{-1} +x(t)),\quad \quad x(t)=\sum_{i\geq 0} x_it^i \in \cfbbb
\end{equation} 
We let $\Op_{\hG}^{\ord_n}$ denote the scheme of such opers.  There is an injective  morphism $\Op_\hG^{\ord_n} \ra \Op_\hG$ sending an $\hN\llb t \rrb $-equivalence class of operators of the above form to the corresponding $\hN \llp t \rrp $-equivalence class. (In fact, this is how one endows $\Op_\hG$ with the structure of an ind-scheme.)

\begin{thm}\cite[Theorem 3.7.9]{BD} \label{t:BD} 
For every $V\in \fg_n-\modd$, the module $\bM_\crit(V)$ is scheme-theoretically supported on $\Op_{\hG}^{\ord_n}$; that is, the natural morphism $\cZ(\hfg_\crit)\ra \End_{\hfg_\crit}(\bM_\crit(V))$ factors through the quotient $\cZ(\hfg_\crit)\simeq \bC[\Op_\hG]\ra \bC[\Op_\hG^{\ord_n}]$. 
\end{thm}

\sssec{Higher residue} 
We now come to the key definition of the present text. To an oper of the form \eqref{eq:oper}, we associate its $n^\mathrm{th}$ residue 
\[
\Res_n (\nabla) := p_{-1}+\ev_n(x(t))=p_{-1}+ \sum_{i=0}^{n-1} x_i t^i,
\]
where $\ev_n: X\llb t \rrb \ra X_n$ denotes the canonical morphism. 
Under conjugation by an element $s \in \hN\llb t \rrb $, the residue gets conjugated by $\ev_n(s) \in \hN_{n}$. (Note that $s^{-1}\partial_t s$ is regular on the disk, and therefore, does not contribute to the higher residue.)
 Thus, the projection of $\Res_n(\nabla)$ onto 
$\cfg_n/\hG_n = \Spec(k[\cfg_n]^{\hG_n})$ is well-defined. According to \cite{RaisTauvel}, the latter algebraic variety is canonically isomorphic to $(\fh^*/W)_n$. Thus, we obtain that $\Res_n$ is a morphism
\[
\Res_n: \Op_{\hG}^{\ord_n} \ra (\fh^*/W)_n.
\]

\sssec{} The notion of residue of opers with regular singularity; i.e., when $n=1$ in \eqref{eq:oper}, is due to Beilinson and Drinfeld \cite[\S 3.8]{BD}. Our residue map $\Res_n$ is an obvious generalization of Beilinson and Drinfeld's residue. On the other hand, in \cite{Laredo}, the authors defined the residue of an oper \eqref{eq:oper} to be equal to $p_{-1}+x(0)\in \cfg/\hG=\fh^*/W$. Therefore, our residue map remembers more information than the one define in \cite{Laredo}; see \S \ref{sss:aside}.

\sssec{} 
For $\chi \in (\fh^*/W)_n$, we denote by $\Op_{G}^{\ord_n, \chi}$ the subscheme of opers with singularity at most $n$ and residue $\chi$. Let $\varpi: \fh^*_n\ra (\fh^*/W)_n$ denote the canonical morphism. Note that each $\Lambda\in \fh_n^*$ defines a point $\varpi(\Lambda)\in (\fh^*/W)_n$ and, therefore, a character of $\cZ(\fg_n)$. We are now ready to state our main result which was conjectured in \cite{oberwolfach}.  

\begin{thm}\label{t:main}
Suppose $Z(\fg_n)$ acts on $V$ by the character $\varpi(\Lambda)$. Then $\bM_\crit(V)$ is scheme-theoretically supported on the subscheme $\Op_{\hG}^{\ord_n, \varpi(-\Lambda)}\subset \Op_\hG$. 
\end{thm}

\sssec{} If one has explicit formulas for the generators of $\cZ(\fg_n)$ and $\cZ(\hfg_\crit)$, then one may prove the above theorem by a direct computation.
For Lie algebras  of type A, the above formulas are given in \cite{MolevTakiff} and \cite{MolevA}, respectively. For types B, C, and D the formulas for generators of $\cZ(\hfg_\crit)$ are provided in \cite{MolevBCD}; however, we do not know of a references which discusses explicit generators for $\cZ(\fg_n)$ for these types. 
The advantage of our method is that it is independent of the type ; in particular, it is valid for exceptional types.

\sssec{} For $n=1$ the above theorem is due to E. Frenkel; see  \cite[\S Proposition 12.6]{FrenkelWakimoto} and \cite[\S 7.6]{BIBLE}. 
In \emph{loc. cit}, this result is proved using the properties of Wakimoto modules. In retrospect, and in view of discussions in \cite{Laredo}, one can give a proof without using the Wakimoto modules. Nevertheless, to establish the above theorem, we find it fruitful to follow the approach of \cite[\S 7.6]{BIBLE}. Namely, we introduce  generalized Verma  and Wakimoto modules and construct a nontrivial morphism  from the Verma to the Wakimoto. This morphism, together with the fact that it is easy to compute the action of $\cZ(\hfg_\crit)$ on Wakimoto modules, will enable us to compute the action of $\cZ(\hfg_\crit)$ on the generalized Verma module. The latter result will, in turn, allow us to establish the above theorem. 

\sssec{Aside: a more general notion of higher residue}\label{sss:aside}
 Let $m$ be a positive integer less than or equal to $n$. For an operator of the form \eqref{eq:oper},   define 
$\displaystyle \Res_n^m(\nabla):=p_{-1}+\sum_{i=0}^{m-1} x_i t^i$.  
Thus, we obtain a morphism 
\[
\Res_n^m: \Op_\hG^{\ord_n} \ra  \cfg_m/\hG_m\simeq (\fh^*/W)_m.
\]
If we take $m=0$, we recover the residue of \cite{Laredo} (mentioned above). 
We expect that a more general version of Theorem \ref{t:main} holds, where one uses $\Res_n^m$ instead of $\Res_n=\Res_n^n$. This would be a generalization of \cite[Theorem 5.6]{Laredo}. We do not discuss this issue in the present text.

\ssec{Generalized Verma and Wakimoto modules} \label{ss:Wakimoto} 
\sssec{Verma modules for jet algebras} \label{sss:Verma} 
We continue to use the notation of the previous subsection, see \S \ref{sss:notation}. Let $\Lambda:\fh_n\ra \bC$ be a character. We extend $\Lambda$ to a character of $\fb_n$ which is trivial on $\fn_n$. Let 
\begin{equation} 
M(\Lambda):=\Ind_{\fb_n}^{\fg_n} (\Lambda) 
\end{equation} 
 The module $M(\Lambda)$ is the Verma modules for $\fg_n$ associated to $\Lambda$. We refer the reader to \cite{Wilson} for more information on these modules.  
 
 \sssec{Central character} 
 Define $\rho_n\in \fh_n^*$ by
\begin{equation}
\rho_n(h_0,\cdots, h_{n-1})=n.\rho(h_0),
\end{equation} 
where $\rho\in \fh^*$ is the sum of simple roots of $\fg$. 
 Let $\varpi:\fh_n^*\ra (\fh^*/W)_n$ denote the canonical morphism. According to \cite[\S 4]{GeoffriauCenter}, under the isomorphism of Theorem \ref{t:jets},  
the centre $Z(\fg_n)$ acts on the Verma module $M(\Lambda)$ by the character $\varpi(\Lambda+\rho_n)\in \Spec(\cZ(\fg_n))$.

\sssec{Generalized Verma modules} Next, let 
\begin{equation} 
\bM_\kappa (\Lambda):= \bM_\kappa(M(\Lambda))= \Ind_{\fgbb\oplus \bC.\bone}^{\hfg_\kappa} (M(\Lambda)).
\end{equation} 
We call $\bM_\kappa(\Lambda)$ a \emph{generalized Verma modules} for the affine Kac-Moody algebra at level $\kappa$. We will now give an alternative construction of these modules. Let 
\[
\fI_n = \ev_n^{-1}(\fb_n), \quad \textrm{and} \quad \fI_n^\circ=\ev_n^{-1}(\fn_n),
\]
where $\ev_n: \fgbb\ra \fg_n$ is the canonical map. 
Note that if $n=1$, then $\fI_n$ is known both as the Iwahori subalgebra and the Borel subalgebra of $\hfg_\kappa$, depending on whether one wants to emphasize the analogy of $\hfg_\kappa$ with $p$-adic groups or with the usual semisimple Lie algebras. We think of $\fI_n$ as ``deeper Iwahori" subalgebras. The character $\Lambda$ defines a character of $\fI_n$ which is trivial on $\fI_n^\circ$. We have an isomorphism of $\hfg_\kappa$ modules 
\[
\bM_\kappa(\Lambda) \simeq \Ind_{\fI_n\oplus \bC.\bone}^{\hfg_\kappa} (\Lambda). 
\]
Our next goal is to relate these Verma modules to certain Wakimoto modules. To define the latter, we need some preparation.

\sssec{Representations of Weyl algebras} \label{sss:Weyl}
Let $\fn^*$ denote the dual of $\fn$. Using the residue pairing, we identity the restricted dual to $\fn\llp t \rrp$ with $\fn^*\llp t \rrp dt$. Thus, we have a canonical non-degenerate inner product on $\fn\llp t \rrp \oplus \fn^*\llp t \rrp dt$. We let $\Gamma$ denote the central of extension of $\fn\llp t \rrp \oplus \fn^*\llp t \rrp dt$ defined by this inner product. 
 The universal enveloping algebra $\cA$ of $\Gamma$ is known as the infinite dimensional  Weyl algebra; see \S\ref{sss:Weyl} for an explicit description of this algebra. Let $\Gamma_{+,n}$ denote the abelian Lie subalgebra $t^{n} \fn\llb t \rrb  \oplus t^{1-n}\fn^*\llb t \rrb dt $. Define 
\begin{equation} 
M_{\fg,n}:=\Ind_{\Gamma_{+,n}\oplus \bC.\bone}^{\Gamma} \, (\bC).
\end{equation} 
In physics literature, the module $M_{\fg,n}$ is sometimes known as the $\beta \gamma$-system of weights $(n, 1-n)$; see, for instance, \cite[Page 53]{Friedan}.

\sssec{Representations of Heisenberg algebras} 
Let $\hfh_\kappa$ denote the affine Kac-Moody algebra at level $\kappa$ associated to the commutative algebra $\fh$. The algebra $\hfh_\kappa$ is known as the affine Heisenberg algebra associated to $\fh$. By an abuse of notation, we let $\Lambda$ also denote the composition $\fhbb\ra \fh_n\rar{\Lambda}\bC$. 
Let 
\begin{equation}\label{eq:Wakimoto}
\pi_\kappa(\Lambda):=\Ind_{\fh\llb t \rrb \oplus \bC.\bone}^{\hfh_{\kappa}} (\Lambda)  
\end{equation} 
Note that the Lie algebra $\hfh_0$ is commutative; in fact, it is isomorphic to $\fhpp\oplus \bC$.

\sssec{Wakimoto modules}  
Let
\begin{equation} 
'\bW_\kappa(\Lambda):=M_{\fg, n} \otimes \pi_{\kappa-\kappa_\crit}(w_0(\Lambda)-2\rho_n), 
\end{equation} 
where $w_0$ is the longest element of the Weyl group $W$. 
Using the free field realization of Feigin and Frenkel \cite[\S 6.1.6]{FrenkelBook}, one can endow $'\bW_\kappa(\Lambda)$ with a canonical action of $\hfg_\kappa$. Following \cite[\S 9.5]{FrenkelBook}, we  twist the action of $\hfg_\kappa$ on $'\bW_\kappa(\Lambda)$ by $w_0$. Henceforth, $\bW_\kappa(\Lambda)$ denotes the module $'\bW_\kappa(\Lambda)$ equipped with this twisted action.\footnote{In more details, $w_0$ defines an involution of $\fg$ which maps $\fn$ to $\fn^-$ and sends $\check{\lambda}\in \fh$ to $w_0(\check{\lambda})$. Therefore, it also defines an involution of $\hfg_\kappa$. The action map $\hfg_\kappa\ra \End(\bW_\kappa(\Lambda))$ is defined to be the composition of this involution with the action map for $'\bW_\kappa(\Lambda)$.} 
We call $\bW_\kappa(\Lambda)$ a generalized Wakimoto module. In the case $n=1$; i.e., when $\Lambda=\lambda$ is a character of $\fh$, the Wakimoto modules defined above have been studied extensively. They are denoted by $\bW_{\kappa,\lambda}^{w_0}$ in \emph{op. cit.}.  

\sssec{Central support of Wakimoto modules}  
\begin{lem} \label{l:Wakimoto} 
The Wakimoto module $\bW_\crit(\Lambda)$ is scheme-theoretically supported on $\Op_{\hG}^{\ord_n, \varpi{(-\Lambda-\rho_n)}}$. 
\end{lem} 
The proof of the above lemma is an easy application of the ultimate form of the Feigin-Frenkel's Isomorphism; see \S \ref{s:Miura}.

\sssec{From Verma to Wakimoto} The following proposition is a first step in understanding the relationship between generalized Verma and Wakimoto modules. 
 \begin{prop} \label{p:morphism}
For every $\Lambda\in \fh_n^*$, there exists a nontrivial homomorphism of $\hfg_\kappa$-modules 
\[
\bM_\kappa(\Lambda) \ra \bW_\kappa(\Lambda).
\]
\end{prop} 
To prove Proposition \ref{p:morphism}, it is enough to prove that the vacuum vector of  $\bW_\kappa(\Lambda)$ is an eigenvector for $\fI_n$ with eigenvalue $\Lambda$. We prove this by an explicit computation, using the  formulas for the action of $\hfg_\kappa$ on $\bW_\kappa(\Lambda)$ given in \cite[Proposition 3.1]{Fedorov}. It may also be possible to give a more conceptual proof of the above lemma, following the reasoning of \cite[\S 11.5]{BIBLE} for the $n=1$ case.

\sssec{}  
In analogy with \cite[\S 9.5]{FrenkelBook}, one can define a Wakimoto module $\bW_\kappa^w(\Lambda)$ for every $w\in W$. (We have only considered the case $w=w_0$). We expect that, at least for generic $\Lambda$, the Verma module
 $\bM_\crit(\Lambda)$ is isomorphic to some Wakimoto module $\bW_\crit^w(\Lambda)$. 
Our motivation is two-fold. On the one hand, 
in the case that $n=1$, one knows that every Verma module is isomorphic to some Wakimoto module; see \cite[\S 11]{BIBLE} and \cite[Proposition 9.5.1]{FrenkelBook}.  On the other hand, the proposed isomorphism between the generalized Verma and Wakimoto modules is a Kac-Moody analogue of an isomorphism due to Bushnell and Kutzko and Roche in the $p$-adic setting. In more details, Wakimoto modules are the analogues of parabolic induction and Verma modules are the analogues of induction from compact open subgroups. The aforementioned result of Bushnell and Kutzko and Roche is an isomorphism between certain parabolically induced representations and their compactly induced counterparts; see, for instance, \cite{Roche}, \cite{BK}, \cite{MasoudTravis1}, and \cite{MasoudTravis2}. 

For $n>1$, we are presently unable to check if the morphism of Proposition \ref{p:morphism} is an isomorphism for any $\Lambda\in \fh_n^*$. The essential difficulty is that we don't know how to define the character of the generalized Verma or Wakimoto modules; thus, we have no way of comparing their sizes.


 \section{Opers with irregular singularities }

\ssec{Recollections on canonical representatives} 
In this subsection, we recall some facts about opers. We continue to use our notation from \S \ref{sss:notation}.
\sssec{}  Let $p_1$ denote the unique element of degree 1 in $\cfn$ such that $\{p_{-1}, 2\rho, p_1\}$ is an $\Sl_2$-triple. Let 
\[
V_\can:=\bigoplus_{i\in E} V_{\can, i}
\]
denote the $\ad\, p_1$-invariants in $\cfn$, decomposed according to the principal gradation. Here $E=\{d_1, \cdots, d_\ell\}$ is the set of exponents of $\cfg$ (so $d_1=1$). As mentioned in \cite[\S 1.3]{BIBLE}, it follows from a theorem of Kostant that the composition 
\begin{equation} \label{eq:Kostant} 
V_\can \rar {v\mapsto v+p_{-1}} \cfg\ra \cfg/\hG \ra \fh^*/W
\end{equation}  
is an isomorphism.

\sssec{}
Choose linear generators $p_j\in V_{\can, d_j}$. If the multiplicity of $d_j$ is greater than one, then we choose linearly independent vectors in $V_{\can, j}$. The following lemma is due to Drinfeld and Sokolov \cite{DS}; see also \cite[Lemma 4.2.2]{FrenkelBook}.

\begin{lem}\label{l:DS} 
The conjugation action of $\hN\llp t\rrp$ on the space of operators of the form \eqref{eq:operDt} is free and each conjugacy class contains a unique operator of the form $\nabla = \partial_t + p_{-1} + v(t)$, where $v(t) \in V_\can\llp t\rrp$, so we can write $\displaystyle v(t) = \sum_{i=1}^\ell v_j(t) p_j$. 
\end{lem} 

\sssec{} It follows from the above lemma that every oper is represented by an $\ell$-tuple $(v_1, \cdots, v_\ell)$ where 
\[
v_i(t) = \sum_{n\in \bZ} v_{i,n} t^{-n-1} \in \bC\llp t\rrp.  
\] 
Therefore, the topological algebra $\bC[\Op_{\hG}]$ is isomorphic to the completion of the polynomial algebra in the variables $v_{i,n}$, $i=1,\cdots, \ell$, $n\in \bZ$, with respect to the topology in which the base of open neighborhoods of zero is formed by the ideals generated by $v_{i,n}$, $n>N$

\sssec{} The analogue of the above lemma also holds for $\Op_\hG^{\ord_n}$; namely, the action of $\hN\llb t\rrb$ on the space of operators of the form \eqref{eq:oper} is free and each conjugacy class contains a unique operator of the form $\nabla = \partial_t + t^{-n} (p_{-1}+v(t))$ where now $v(t)\in V_\can\llb t\rrb$. Equivalently, every element $\nabla \in \Op_\hG^{\ord_n}$ has a canonical representative of the form 
\[
\nabla = \partial_t + t^{-n} \left( p_{-1} + \sum_{j=1}^\ell c_j(t) p_j\right), \quad \quad c_j(t) \in \bC\llb t\rrb. 
\]

\sssec{} We will now investigate the map $\Op_\hG^{\ord_n}\ra \Op_\hG$, following \cite[Proposition 4.2]{Laredo}. Applying the conjugation by $\rho(t)^{-n}$ to the operator $\nabla$, we obtain 
\[
\partial_t + p_{-1} + n\rho t^{-1} + \sum_{j=1}^\ell t^{-n(d_j+1)} c_j(t) p_j.
\]
Next, applying conjugation by $\exp(-np_1/2t)$, we obtain the operator 
\[
\partial_t+ p_{-1}+ \left(t^{-n-1} c_1(t) - \frac{n^2-2}{4} t^{-2}\right) p_1 + \sum_{j>1} t^{-n(d_j+1)} c_j(t) p_j . 
\]

\sssec{} The following result is  essentially due to Drinfeld and Sokolov \cite[Proposition 3.8.9]{DS}.

\begin{prop} \label{p:DS}
The natural morphism $\Op_\hG^{\ord_n} \ra \Op_\hG$ is 
injective. Its image consists of opers whose canonical representatives  have the form 
\[
\nabla = \partial_t + p_{-1} + \sum_{i=1}^\ell t^{-n(d_j+1)} u_j(t) p_j,\quad \quad u_j(t) \in \bC\llb t\rrb. 
\] 
Moreover,
\[
\Res_n(\nabla) = \left(\ev_n(u_1(t)) + \frac{n^2-2}{4} t^{n-1}\right)p_1 + \sum_{j>1} \ev_n(u_j(t)) \in (V_\can)_n\simeq (\fh^*/W)_n
\]
where the last isomorphism is given by \eqref{eq:Kostant}. 
\end{prop}

\sssec{} Let us identify $\bC[\Op_\hG]$ with the completion of a polynomial algebra generated by $\{v_{i,m}\}$ where $i=1,\cdots, \ell$, and $m\in \bZ$. The following corollary follows immediately from the above proposition.

\begin{cor}  \label{c:rep}
\begin{enumerate} 
\item[(i)] The ideal corresponding to the quotient $\bC[\Op_\hG]\ra \bC[\Op_\hG^{\ord_n}]$ is topologically generated by the elements $v_{i,n_i}$, where $n_i\geq n(d_i+1)$. 
\item[(ii)] The algebra  $\bC[\Op_\hG^{\ord_n}]$ is isomorphic to the polynomial algebra generated by 
$v_{i,n_i}$,  where $n_i< n(d_i+1)$. 
\item[(iii)] The algebra $\cR_n\subseteq \bC[\Op_\hG^{\ord_n}]$ generated by the elements 
$v_{i,n_i}$, $nd_i \leq n_i<n(d_i+1)$, is isomorphic to $\bC[(V_\can)_n]\simeq \bC[(\fh^*/W)_n]$. Moreover, the resulting embedding 
$\bC[(\fh^*/W)_n] \ra \bC[\Op_\hG^{\ord_n}]$
is the homomorphism $\Res_n^*$.  
\item[(iv)] The injective homomorphism $\Res_n^*: \bC[(\fh^*/W)_n]\ra \Op_\hG^{\ord_n}$ is an isomorphism onto $\cR_n$. 
 \end{enumerate} 
 \end{cor} 

\ssec{Principal symbols of the central elements} 

\sssec{}\label{sss:enveloping}
 Let $U_\kappa(\hfg):=U(\hfg_\kappa)/(\bone-1)$. Define a linear topology on $U_\kappa(\hfg)$ by using as the basis of neighborhoods of identity the following left ideals:
\[
I_N=U_\kappa(\hfg) (\fg\otimes t^N \bCbb), \quad \quad N\geq 0. 
\] 
Let $\tU_\kappa(\hfg)$ denote the completion of $U_\kappa(\hfg)$ with respect to this topology. Equivalently $\tU_\kappa(\hfg)$ is the inverse limit of $U_\kappa(\hfg)/I_N$. The category of $\tU_\kappa(\hfg)$-modules is then canonically identified with the category of smooth modules over $\hfg$ at level $\kappa$. In particular, the Feigin-Frenkel centre $\cZ(\hfg_\crit)$ identifies with the centre of the topological associative algebra $\tU_\crit(\hfg)$. 

\sssec{} Next, note that PBW Theorem defines a canonical filtration on $\tU_\kappa(\hfg)$, and therefore, on $\cZ$.   To describe the associated graded, we need some notation. Choose a basis $\{ J^a\}$ for $\fg$ and let $\{\bJ^a\}$ denote the dual basis for $\fg^*$.
By Chevalley's restriction theorem, the algebra $\bC[\fg^*]^\fg$ is isomorphic to the polynomial algebra $\bC[\bP_i]_{i=1,\cdots, \ell}$, where $\bP_i$ may be chosen to be homogeneous of degree $d_i+1$. Let $\bP_{i,m}\in \bC[\fg^*\llp t\rrp]$ denote the polynomials defined by 
\[
\bP_{i} (\bJ^a(z)) =\sum_{m\in \bZ} \bP_{i,m} z^{-m-1}, 
\]
where $\displaystyle \bJ^a(z)=\sum_{m\in \bZ} \bJ_m^a.z^{-m-1}$ and $\bJ_m^a$ are generators of $\bC[\fg^*\llp t\rrp]$, defined by the formula 
\[
\bJ_m^a(\phi(t))= \Res_{t=0} \langle \phi(t), J^a\rangle t^m dt, \quad \quad \phi(t) \in \fg^*\llp t\rrp, n\in \bZ. 
\]

\sssec{} Combining results of Feigin and Frenkel \cite{FF} and Beilinson and Drinfeld \cite{BD}, one can show that the associated graded  of $\cZ_\crit$ is isomorphic to a completion of the polynomial algebra 
\[
\bC[\bP_{i,m}], \quad i=\{1,\cdots, \ell\},\quad m\in \bZ,
\]
see, for instance, \cite[\S 5.2]{Laredo}. Let $S_{i,m}\in \cZ(\hfg_\crit)$ denote the image of $v_{i,m}$ under the Feigin-Frenkel Isomorphism $\bC[\Op_\hG] \simeq \cZ(\hfg_\crit)$. Then, one the principal symbol of $S_{i,m}$ equals $\bP_{i,m}$.

\sssec{} For the rest of this section, we let $\cZ:=\cZ(\hfg_\crit)$. 
Let $\cZ^{\ord_n}$ denote the quotient of $\cZ(\hfg)$ corresponding to $\bC[\Op_\hG^{\ord_n}]$ under the Feigin-Frenkel Isomorphism $\cZ \simeq \bC[\Op_\hG]$. By Corollary \ref{c:rep}, the algebra $\cZ^{\ord_n}$ is isomorphic to a polynomial algebra on 
\begin{equation} \label{eq:central1}
S_{i,[m]},\quad i=1,\cdots, \ell,\quad m < n(d_i+1).
\end{equation}

\sssec{} We now consider the action of $S_{i,[m]}$ on the the universal module $\bU_n:=\bM_\crit(U(\fg_n))$. 
Note
\[
\bU_n:=\bM_\crit(U(\fg_n)) \simeq \Ind_{t^n \fgbb}^{\hfg_\crit} (\bC). 
\] 
In particular, the generating vector $x_n\in \bU_n$ satisfies the property $t^n\fgbb.x_n=0$. As mentioned above $S_{i,[m]}$ equals $\bP_{i,m}$ plus lower order terms. 

\begin{prop} 
\begin{enumerate} 
\item[(i)] The elements $S_{i,[m]}$, $m\geq n(d_i+1)$ annihilate $x_n$. 
\item[(ii)] The elements $S_{i,[nd_i+k]}$, $k=0, \cdots, n-1$, act on $x_n$  by their symbol $\bP_{i,nd_i+k}$. 
\end{enumerate} 
\end{prop} 

\begin{proof}
Part (i) is a restatement of Theorem \ref{t:BD}; see \cite[Theorem 5.6.(1)]{Laredo}  for a reproduction of the proof. Part (ii) also follows from \cite[Lemma 3.9]{Laredo}. For $k=n-1$ this is discussed in the proof Theorem 5.6 of \emph{loc. cit}. The reasoning employed there applies to all $0\leq k\leq n-1$. 
\end{proof} 

As an immediate corollary of Part (ii), we have: 
\begin{cor} \label{c:preserve} 
The elements $S_{i,[nd_i+k]}$, $k=0, \cdots, n-1$, preserve $U(\fg_n) \subset \bM(U(\fg_n))$. 
\end{cor}

\ssec{Proof of Theorem \ref{t:main}}  
Let $V\in \fg_n-\modd$. In this section, we write $\bM(V)$ for the module $\bM_\crit(V)\in \hfg_\crit-\modd$. 
We already know that the action of $\cZ$ on $\bM(V)$ factors through $\cZ^{\ord_n}$. Next, observe that 
 $\bC[(\fh^*/W)_n]$  acts on $\bM(V)$ by the composition 
\begin{equation}\label{eq:action}
 \bC[(\fh^*/W)_n]\rar{\Res_n^*} \bC[\Op^{\ord_n}]\simeq \cZ^{\ord_n}\ra \End_{\hfg}(\bM(V)).
 \end{equation} 
  According to Corollaries \ref{c:rep} and \ref{c:preserve}, the action map \eqref{eq:action} comes from a homomorphism 
\begin{equation}\label{eq:action}  
\bC[(\fh^*/W)_n]\rar{\Res_n^*} \cR_n \simeq \cS_n \ra \cZ(\fg_n)\ra \End_{\fg_n} (V) \ra \End_{\hfg}(\bM(V)).
\end{equation}
In particular, this means that under this action $\bC[(\fh^*/W)_n]$ acts on a Verma module $\bM(\Lambda)$ by a character, a fact that is not obvious without Corollary \ref{c:preserve}. 
Let us compare the above homomorphism $\bC[(\fh^*/W)_n] \ra \cZ(\fg_n)$ with the Harish-Chandra's Isomorphism for jet algebras given in Theorem \ref{t:jets}.

 \sssec{} To do the comparison, it is enough to consider the Verma modules $V=M(\Lambda)$. As mentioned in \S \ref{sss:Verma}, Geoffriau has proved that $\cZ(\fg_n)\simeq \bC[(\fh^*/W)_n]$ acts on $M(\Lambda)$ via the character $\varpi(\Lambda+\rho_n)$ \cite[\S 4]{GeoffriauCenter}.  Now by Lemma \ref{l:Wakimoto}, the character of $\bC[(\fh^*/W)_n]$ acting on $\bW(\Lambda)$ via 
\[
  \bC[(\fh^*/W)_n] \simeq \cR_n \simeq  \cS_n \subseteq \cZ^{\ord_n}\ra \End_{\hfg}(\bW(\Lambda)),
\]
  is given by $\varpi(-\Lambda-\rho_n)$. By Proposition \ref{p:morphism}, we have a nontrivial homomorphism $\bM(\Lambda)\ra \bW(\Lambda)$; hence, $\bC[(\fh^*/W)_n]$ acts on both modules by the same character. Thus, we conclude that $\bC[(\fh^*/W)_n]$ acts on $\bM(\Lambda)$ by the character $\varpi(-\Lambda-\rho_n)$.  
  
  \sssec{} 
  Thus, the homomorphism $\bC[(\fh^*/W)_n] \ra \cZ(\fg_n)$ of \eqref{eq:action} equals the composition of the  Harish-Chandra Isomorphism $\bC[(\fh^*/W)_n] \ra \cZ(\fg_n)$ and the automorphism of $\bC[(\fh^*/W)_n]$ induced by $\varpi(\Lambda)\ra \varpi(-\Lambda)$. (Note that $\varpi: \fh_n^*\ra (\fh^*/W)_n$ is a dominant morphism, so indeed there is a unique such automorphism of $(\fh^*/W)_n$.)
 Therefore, if $\cZ(\fg_n)$ acts on $V$ by the character $\varpi(\Lambda)$, then $\bC[(\fh^*/W)_n]$ acts, via \eqref{eq:action}, on $\bM(V)$ by the character $\varpi(-\Lambda)$. In other words, $\bM(V)$ is scheme-theoretically supported on $\cZ^{\ord_n, \varpi(-\Lambda)}$, as required.


\section{Irregular connections and Miura Transformation} \label{s:Miura}
In this section, we use the ultimate form of the Feigin-Frenkel Isomorphism to compute the central support of generalized Wakimoto modules; in particular, we will prove Lemma \ref{l:Wakimoto}. We continue to use the notation of the introduction, see \S \ref{sss:notation}. 

\ssec{Ultimate form of the Feigin-Frenkel Isomorphism} 
\sssec{} Let $\Omega^{\rho}$ denote the $\hH$-bundle defined as the push-forward of the $\bC^\times$-bundle corresponding to the canonical line bundle $\Omega$ on the punctured disk $\cDt$ under the homomorphism $\rho: \bC^\times\ra \hH$. We let $\Conn_\hH(\Omega^{\rho})$  denote the space of connections on $\Omega^{\rho}$.  A choice of coordinate $t$ on the disk gives rise to a trivialization of $\Omega$, and hence of $\Omega^{\rho}$. A connection on $\Omega^{\rho}$ may then be written as an operator 
\begin{equation} \label{eq:miura}
\nabla = \partial_t + u(t), \quad \quad u(t) \in \fh^* \llp t \rrp .
\end{equation}

\sssec{} The \emph{Miura Transformation} is defined to be the morphism 
\[
\MT: \Conn_\hH(\Omega^{\rho}) \ra \Op_{\hG}
\]
 which sends an operator of the form \eqref{eq:miura} to the $\hN \llp t \rrp $-conjugacy class of $p_{-1}\plus \nabla$. We refer the reader to \cite[\S 9.3.1]{FrenkelBook} for more details on this map.

\sssec{} Similarly, we define the $\hH$-bundle $\Omega^{-\rho}$  and connections on this bundle. Choosing a local coordinate, we can write these connections as 
\begin{equation} \label{eq:miura2} 
\nabla = \partial_t + b(t), \quad \quad b(t) \in \fh^* \llp t \rrp .
\end{equation} 
In particular, the algebra $\bC[\Conn_\hH(\Omega^{-\rho})]$ is isomorphic to $\fhpp dt$. Thus, smooth $\fhpp$-modules are in bijection with smooth modules over $\bC[\Conn_\hH(\Omega^{-\rho})]$.

\sssec{Ultimate form of the Feigin-Frenkel Isomorphism} Let $M$ be a smooth  $\Gamma$-module and let $R$ be a smooth  $\bC[\Conn_\hH(\Omega^{-\rho})]$-module. Using the free field realization formalism, one can define a structure of a $\hfg_\crit$-module on $M\otimes R$ \cite[\S 6.1.6]{FrenkelBook}.
We Let $\theta: \Conn_\hH(\Omega^{-\rho})\ra \Conn_\hH(\Omega^{\rho})$ denote the isomorphism given, in local coordinate, by  
\begin{equation}\label{eq:theta}
\theta(\partial_t+b(t))= \partial_t -u(t)
\end{equation}

\begin{thm}\cite[Theorem 8.3.3]{FrenkelBook} \label{t:FFUltimate} 
 \begin{enumerate} 
 \item[(i)] The action of $\cZ(\hfg_\crit)$ on $M\otimes R$ does not depend on $M$ and factors through a homomorphism $\cZ(\hfg_\crit) \ra \bC[\Conn_\hH(\Omega^{-\rho})]$ and the action of $\bC[\Conn_\hH(\Omega^{-\rho})]$ on $R$. 
 \item[(ii)] The corresponding morphism $\Conn_\hH(\Omega^{-\rho})_{\cDt}\ra \Spec(\cZ(\hfg_\crit))$ fits into a commutative diagram 
\[ 
\xymatrix{
\Conn_\hH(\Omega^{-\rho}) \ar[d] \ar[r]^\theta & \Conn_\hH(\Omega^{\rho}) \ar[d]^{\MT} \\
\Spec(\cZ(\hfg_\crit)) \ar[r] & \Op_{\hG}. 
}
\]
\end{enumerate} 
\end{thm} 

The above theorem implies that the computation of the central support of the (potentially complicated) module $M\otimes R$ is reduced to a simpler computation involving only $R$.

\ssec{Higher residue for irregular Miura opers} 
\sssec{} 
Let $\Conn_\hH^{\ord_n}(\Omega^{\rho})\subset \Conn_\hH(\Omega^{\rho})$ denote the subspace of connections with singularities of order at most $n$ on  $\Omega^{\rho}$. The elements of $\Conn_\hH^{\ord_n}$
are represented by operators of the form 
\begin{equation}\label{eq:MiuraSing}
\nabla = \partial_t + t^{-n} u(t), \quad \quad u(t)=\sum_{i \geq 0} u_i t^i \in \fh^*\llb t \rrb . 
\end{equation}

Let $\Res_{\fh_n^*}^\rho: \Conn_\hH^{\ord_n}(\Omega^{\rho}) \ra \fh_n^*$ by 
\begin{equation}
\Res_{\fh_n^*}^\rho(\nabla):=\sum_{i=0}^{n-1} u_i t^i.
\end{equation} 
For every $\Lambda\in \fh_n^*$, we let $\Conn_\hH^{\ord_n, \Lambda}(\Omega^{\rho})$ denote the subspace of $\Conn_\hH^{\ord_n}(\Omega^{\rho})$ consisting of the elements of residue $\varpi(\Lambda)$, where $\varpi: \fh_n*\ra (\fh^*/W)_n$ is the canonical morphism.

\sssec{} In a similar fashion,  we have a subspace $\Conn_\hH^{\ord_n}(\Omega^{-\rho})\subset \Conn_\hH(\Omega^{-\rho})$ and a residue map $\Res_{\fh_n^*}^{-\rho}:  \Conn_\hH^{\ord_n}(\Omega^{-\rho}) \ra \fh_n^*$ and the subspace
$\Conn_\hH^{\ord_n, \Lambda}(\Omega^{-\rho})$. Note that under the isomorphism $\theta$ of \eqref{eq:theta}, we have 
\[
\Conn_\hH^{\ord_n, -\Lambda}(\Omega^{\rho})\simeq \Conn_\hH^{\ord_n, \Lambda}(\Omega^{-\rho})
\]
Explicitly, $\Conn_\hH^{\ord_n, \Lambda}(\Omega^{-\rho})$ consists of connections of the form 
\[
\partial_t + t^{-n} b(t)
\]
where $b(t)\in \fh^*\llb t \rrb$ satisfies $\ev_n(b(t))=\Lambda$. (For the analogous statement in the case $n=1$, see \cite[\S 9.4.3]{FrenkelBook}.) In particular, under the identification $\Conn_\hH(\Omega^{-\rho})\simeq \fhpp$, we have an isomorphism 
\begin{equation} \label{eq:piLambda}
\bC[\Conn_\hH^{\ord_n, \Lambda}(\Omega^{-\rho})]\simeq \pi_0(\Lambda)=\Ind_{\fhbb\oplus \bC.\bone }^{\hfh_0} (\Lambda). 
\end{equation} 
In fact, the LHS is a coordinate independent version of the RHS.

\sssec{}
 We  compute how the residue changes under the Miura Transformation. The analogous computation for opers with regular singularities appears in \cite[\S 9.3.1]{FrenkelBook}
\begin{lem} \label{l:Miura}
 The following diagram is commutative 
\[
\xymatrix{ 
\Conn_\hH^{\ord_n}(\Omega^\rho) \ar[r]^{\quad \MT} \ar[d]_{\Res_{\fh_n^*}^\rho} & \Op_G^{\ord_n}\ar[d]^{\Res_n}\\
\fh_n^* \ar[r] & (\fh^*/W)_n,
}  
\]
where the bottom map is the composition of $\Lambda \mapsto \Lambda-\rho_n$ and the projection $\varpi: \fh_n^* \ra (\fh^*/W)_n$. 
\end{lem}

\begin{proof} Miura transformation is given by sending an operator of the form \eqref{eq:MiuraSing} to the $\hN \llp t \rrp $-conjugacy classes of the operator $\partial_t + p_{-1} + t^{-n} u(t)$. Applying conjugation with $\rho(t^n)$, we identify it with the $\hN \llp t \rrp $-equivalence class of the operator $\partial_t + t^{-n} (p_{-1} - n. \rho. t^{n-1} + u(t \rrp $, whose residue equals $\varpi(\sum_{i=0}^{n-1}u_i t^i -\rho_n)$. 
\end{proof}

\ssec{Proof of Lemma \ref{l:Wakimoto}}
Our goal is to compute the central support of $\bW_\crit(\Lambda)$. For the case of $n=1$ of this computation, that is, when $\Lambda$ is a character of $\fh$, see \cite[\S 9.4.3]{FrenkelBook}. Recall that the underlying vector space of this module is equal to $'\bW_\crit(\Lambda)=M_{\fg,n}\otimes \bC[\Conn_\hH^{\ord_n, \Lambda}(\Omega^{w_0(\Lambda)-2\rho_n})]$ but the action of $\hfg_\crit$ has been twisted by $w_0$. 

\sssec{} Note that for any ring $R$ and an ideal $I\subseteq R$, $\End_R(R/I)=R/I$. In particular, the canonical morphism $R\ra \End_R(R/I)$ has a factorization $R\ra R/I = \End_R(R/I)$. 
Applying this obvious fact to our situation, we conclude that the canonical morphism $\bC[\Conn_\hH(\Omega^{-\rho})] \ra \End_{\bC[\Conn_\hH(\Omega^{-\rho})]}(\bC[\Conn_\hH^{\ord_n,\Lambda}(\Omega^{-\rho})])$ has a factorization 
\[
\bC[\Conn_\hH(\Omega^{-\rho})]\ra \bC[\Conn_\hH^{\ord_n,\Lambda}(\Omega^{-\rho})] \ra \End_{\bC[\Conn_\hH(\Omega^{-\rho})]} (\bC[\Conn_\hH^{\ord_n,\Lambda}(\Omega^{-\rho})])
\]

\sssec{} In view of Lemma \ref{l:Miura}, the restriction of the Miura transformation to $\Conn_\hH^{\ord_n,\Lambda}(\Omega^{-\rho})\simeq \Conn_\hH^{\ord_n,-\Lambda}(\Omega^{\rho})$ takes values in 
\[
\Op_{\hG}^{\ord_n, \varpi(-\Lambda-\rho_n)} \subset \Op_\hG
\]

\sssec{} Combining the previous observation with Theorem \ref{t:FFUltimate}.(i), and using the equality 
\[
-(w_0(\Lambda)-2\rho_n)-\rho_n=-w_0(\Lambda)-\rho_n=-w_0(\Lambda+\rho_n),
\]
we conclude that the action morphism $\cZ(\hfg_\crit)\ra \End('\bW_\crit(\Lambda))$ has a factorization of the form 
\[
\cZ(\hfg_\crit) \ra \bC[\Op_{\hG}^{\ord_n, -w_0(\Lambda+\rho_n)}]\ra \End('\bW_\crit(\Lambda)).
\]

\sssec{}
Finally, recall that we are interested in the twisted action of $\cZ(\hfg_\crit)$ on the Wakimoto modules, where the twisting is given by $w_0$. Therefore, $\cZ(\hfg_\crit)$ acts on $\bW_\crit(\Lambda)$ via the quotient $\cZ(\hfg_\crit)\simeq \bC[\Op_{\hG}] \ra \bC[\Op_\hG^{\ord_n, -\Lambda-\rho_n}]$, as required.


\section{Action of the affine algebra on generalized Wakimoto modules} 
In this section, we recall the formulas for action of $\hfg_\kappa$ on Wakimoto modules and use this to construct a morphism from the Verma to the Wakimoto module. Recall that we have fixed a triangular decomposition $\fg=\fn^-\oplus \fh\oplus \fn$. Let $\Delta\subset \fh^*$ be the root system of $(\fg,\fh)$ and let $\Delta_+\subset \Delta$ denote the positive roots.  Denote by $\alpha_1,\cdots, \alpha_\ell$ the simple roots. For every $\alpha\in \Delta_+$, fix an $\Sl_2$-triple $\{e_\alpha, f_\alpha, h_\alpha\}$. We let $h_i=h_{\alpha_i}$, $i=1, \cdots, \ell$. 

\ssec{Explicit realization of Wakimoto modules} 
\sssec{} \label{sss:Weyl} 
The Weyl algebra $\cA$, defined in \S \ref{ss:Wakimoto} as the universal enveloping algebra of $\Gamma$, is isomorphic to the associative algebra  generated by $a_{\alpha,n}$ and $a_{\alpha,n}^*$, $\alpha\in \Delta_+$, $n\in \bZ$, subject to the relations 
\[
[a_{\alpha,n}, a_{\alpha,m}^*]= \delta_{\alpha,\beta} \delta_{n,-m},\quad [a_{\alpha,n}, a_{\beta,m}]=[a_{\alpha,n}, a_{\beta,m}]=0.
\]
\sssec{} By definition, $M_{\fg,n}$ is the representation of $\cA$ generated by a vector $\bu_n$ satisfying the properties  
\begin{equation}\label{eq:Mgn}
a_{\alpha, m}.\bu_n =0,  \quad   m\geq n \quad \quad \textrm{and} \quad \quad a_{\alpha, m}^*.\bu_n=0 \quad   m \geq 1-n.
\end{equation}

\sssec{} The Heisenberg algebra $\hfh_\kappa$ is isomorphic to the Lie algebra generated by $b_{i,n}$, $i=1, \cdots, \ell$, $n\in \bZ$, and a central element $\bone$, subject to the relation 
\[
[b_{i,n}, b_{j,m}]= -n \kappa(h_i, h_j) \delta_{n,-m} \bone. 
\] 

\sssec{} Let $\Lambda\in \fh_n^*$ and write $\Lambda=(\lambda_1, \cdots, \lambda_n)$, where $\lambda_i\in \fh^*$. The  module $\pi_\kappa(\Lambda)$ is the representation of $\hfh_\kappa$ generated by a vector $\bv_n$ satisfying 
\begin{equation}\label{eq:piLambda}
b_{i,m}.\bv_n=
\begin{cases} 
0  &  m\geq n \\
\langle \lambda_m, h_i \rangle \bv_n & 0\leq m<n. 
\end{cases} 
\end{equation}

\sssec{} Our goal is to give formulas for the action of $\hfg_\kappa$ on the tensor product $M_{\fg,n}\otimes \pi_{\kappa-\kappa_\crit}(\Lambda)$. To this end, define the generating functions 
\[
a_\alpha(z) := \sum_{n\in \bZ} a_{\alpha, n} z^{-n-1}, \quad a_\alpha^*(z) := \sum_{n\in \bZ} a_{\alpha, n}^* z^{-n}, 
\quad b_i(z) := \sum_{n\in \bZ} b_{i, n} z^{-n-1}, 
\]
\[
e_\alpha(z) := \sum_{n\in \bZ} e_{\alpha, n} z^{-n-1}, \quad
h_i(z) := \sum_{n\in \bZ} h_{i, n} z^{-n-1}, \quad
f_\alpha(z) := \sum_{n\in \bZ} f_{\alpha, n} z^{-n-1}. 
\]

For a monomial $A$ in $a_{\alpha,n}$ and $a_{\alpha,m}^*$, we define its normal ordering $:A:$ by moving all $a_{\alpha, n}$ with $n\geq 0$ and all $a_{\alpha,m}^*$ with $m>0$ to the right. The following proposition appears in \cite[Proposition 3.1]{Fedorov}; it is a, slightly sharper, variant of the formulas of Feigin and Frenkel; see \cite[\S 6]{FrenkelBook}
\begin{prop} \label{p:formulas}
The action of $\hfg_\kappa$ on $M_{\fg,n}\otimes \pi_{\kappa-\kappa_\crit}(\Lambda)$ is given by 
\[
f_\alpha(z) \mapsto \sum_{\beta\in \Delta_+} :Q_\beta^\alpha(a^*(z))a_\beta(z): + 
\sum_{\beta\in \Delta_+} \tilde{Q}_\beta^\alpha(a^*(z)) \partial_z a_\beta^*(z) + b_\alpha(z) a_\alpha^*(z) + \sum_i b_i(z) R_i^\alpha(a^*(z)), 
\]
\[
h_i(z) \mapsto \sum_{\beta \in \Delta_+} \beta(h_i) -:a_\beta^*(z) a_\beta(z): + b_i(z),
\]
\[
e_\alpha(z) \mapsto a_\alpha(z) + \sum_{\beta\in \Delta_+, \beta>\alpha} :P_\beta^\alpha(a^*(z))a_\beta(z):,
\]
where  $P_\beta^\alpha$, $Q_\beta^\alpha$ and $R_i^\alpha$ are polynomials in $a_\gamma^*$ such that $P_\alpha^\beta$'s and $Q_\alpha^\beta$'s have no constant terms and $R_i^\alpha$'s have no constant and linear terms. 
\end{prop}

\ssec{Elementary computations} 
\sssec{} In what follows, for a generating function $X(z)=\sum_n X_n z^{-n-1}$, we let $[X(z)]_n := X_n$. Thus, for instance, for $m\geq 0$, we have 
\begin{equation}\label{eq:normal}
[:a_\gamma^*(z)a_\beta(z):]_m= \sum_{r+s=m,\,r\leq 0} a_{\gamma,r}^* a_{\beta,s} + \sum_{r+s, \, r>0} a_{\beta,s} a_{\gamma, r}^*.
\end{equation}

\sssec{} We need a few results about normally ordered products acting on the vacuum vector $\bu_n\in M_{\fg,n}$. 
\begin{lem} \label{l:basic0} 
For all $\beta\in \Delta_+$, 
\begin{equation} 
[:a_\beta(z)a_\beta^*(z):]_0.\bu_n=n\bu_n, \quad \textrm{and} \quad [:a_\beta(z)a_\beta^*(z):]_m.\bu_n=0, \quad \forall m>0. 
\end{equation} 
\end{lem} 
\begin{proof}
 Indeed, for $m\geq 0$, we have 
\[
[:a_\beta(z)a_\beta^*(z):]_m.\bu_n = \sum_{r+s=m,\,r\leq 0} a_{\beta,r}^* a_{\beta,s}.\bu_n + \sum_{r+s=m, \, r>0} a_{\beta,s} a_{\beta, r}^*.\bu_n = \sum_{r\leq 0} a_{\beta,r}^* a_{\beta,m-r}.\bu_n.
\]
Note that all the summands of the RHS are zero except when $m-r=0, 1, \cdots, n-1$.  

If $m=0$, the sum equals $\displaystyle \sum_{-r=0}^{n-1} (1-a_{\beta, -r}a_{\beta,r}^*).\bu_n$, which equals $n\bu_n$, since $a_{\beta,r}^*.\bu_n=0$ for $r=0, -1, \cdots, 1-n$. Thus, we have established the first equality in the lemma. 

Next, note that if $m>0$, then $-n+m+1\leq r\leq m$; in particular, $r\geq 1-n$. Moreover, $a_{\beta,r}^*$ and  $a_{\beta,m-r}$ commute and $a_{\beta,r}^*.\bu_n=0$ for such $r$'s. The second equality is, therefore, established. 
\end{proof}

\begin{lem} \label{l:basic} 
Let $T(a^*(z))$ be a polynomial in $a_\gamma^*(z)$, $\gamma\in \Delta_+$, which has no constant term. Let $\beta\in \Delta_+$. 
\begin{enumerate} 
\item[(i)] For all $m\geq n$, we have $[:T(a^*(z))a_\beta(z):]_m.\bu_n=0$. 
\item[(ii)] If, in addition, $T$ has no summand of the form $ca_\beta^*(z)$, where $c\in \bC$, then we have
$[:T(a^*(z))a_\beta(z):]_m.\bu_n=0$ for all $m\geq 0$.  
\end{enumerate} 
\end{lem} 

\begin{proof} Formula \eqref{eq:normal} implies immediately that (i) holds if $T$ is linear; that is, if $T(a^*(z))=a_\gamma^*(z)$. The proof for higher degree polynomials is similar. 

 For (ii), first of all observe that if $\gamma\neq \beta$, then in view of the fact that $a_{\gamma,n}^*$ and $a_{\beta,m}$ commute for all $m,n\in \bZ$, Equation \eqref{eq:normal} shows that 
  \begin{equation}\label{eq:r0}
  [a_\gamma^*(z) a_\beta(z)]_m.\bu_n=0, \quad \quad \forall m\geq 0. 
  \end{equation}
   Therefore, we are reduced to proving the following statement: 
   \[
   \textrm{If $T$ has no constant or linear terms, then $[T(a^*(z))a_\beta(z)]_m.\bu_n=0$, for all $m\geq 0$.}
   \]
    Let us prove the above fact for a quadratic polynomial $T(a^*(z))=a_\alpha^*(z)a_\gamma^*(z)$. The general case follows by similar considerations. For ease of notation, we also omit the subscripts $\alpha$, $\beta$, etc., and write simply $a^*$ for $a_\alpha^*(z)$, etc.  Thus, we have   
\[
:a^*a^*a:= :a^* (a_+^*a +aa_-^*): =  a_+^*(a_+^*a +aa_-^*)+(a_+^*a + aa_-^*) a_-^* 
\]
where $a_+=\sum_{n<0} a_n z^{-n-1}$ and $a_+^* = \sum_{m\leq 0} a_m^* z^{-m}$ and $a_-$ and $a_-^*$ are defined as the complements. Thus, 
\[
[:a^*a^*a:]_m.\bu_n = \sum_{r+s+p=m, r,s\leq 0} a_r^* a_s^* a_p \bu_n
\] 
Now note that by our assumption $\gamma\neq \beta$; therefore, $a_s^*$ and $a_p$ always commute. Moreover, the fact that $s+p=m-r\geq 0$ implies that 
either $s\geq 1-n$ or $p\geq n$, thus one of $a_s^*$ or $a_p$ must kill $\bu_n$. We conclude that the above sum always equals zero. 
\end{proof}

\sssec{} We now study the polynomials $Q_\beta^\alpha$ appearing in Proposition \ref{p:formulas}. 
 
\begin{cor} \label{c:basic} 
Write $Q_\beta^\alpha(a^*(z))= T_\beta^\alpha(a^*(z))+c_\beta a_\beta^*(z)$ where $T_\beta^\alpha$ is a polynomial in $a_\gamma^*(z)$ and has no summand of the form $ca_\beta^*(z)$, where $c$ is a scalar. Then, we have 
\begin{enumerate} 
\item[(i)] $\sum_{\beta\in \Delta_+} c_\beta = 0$. 
\item[(ii)] For all $\alpha\in \Delta_+$, we have $f_{\alpha,0}.\bu_n=0$. 
\end{enumerate} 
\end{cor} 

\begin{proof} Note that 
\[
f_{\alpha,0}.\bu_n=  \sum_{\beta\in \Delta_+} [:Q_\beta^\alpha(a^*(z))a_\beta(z):]_0.\bu_n + 
\sum_{\beta\in \Delta_+} [\tilde{Q}_\beta^\alpha(a^*(z)) \partial_z a_\beta^*(z)]_0.\bu_n 
\]
Write $Q_\beta^\alpha(a^*(z))= T_\beta^\alpha(a^*(z))+c_\beta a_\beta^*(z)$ where $T_\beta^\alpha$ is a polynomial in $a_\gamma^*(z)$ and has no summand of the $a_\beta^*(z)$. By the previous lemma $[:T_\beta^\alpha(a^*(z)) a_\alpha(z):]_0.\bu_n=0$. It is also easy to see that $[\tilde{Q}_\beta^\alpha(a^*(z)) \partial_z a_\beta^*(z)]_0$ annihilates $\bu_n$. Therefore, 
\[
f_{\alpha,0}.\bu_n= \sum_{\beta\in \Delta_+} c_\alpha [:a_\beta(z) a_\beta^*(z):]_0.\bu_n=n.(\sum_{\beta\in \Delta_+} c_\beta)\bu_n
\]
According to \cite[Proposition 6.3.1]{FrenkelBook}, $\fn^-$ annihilates the vector $\bu_1$; that is, $f_{\alpha,0}.\bu_1=0$. It follows that $\sum_{\beta\in \Delta_+} c_\beta=0$, as required.  
\end{proof}

\sssec{} We now consider the action of certain expressions on the vacuum vector $\bw_n:=\bu_n\otimes \bv_n \in \bM_{\fg,n}\otimes \pi(\Lambda)$.
\begin{lem} \label{l:basic2}
Let $R(a^*(z))$ be a polynomial with no constant terms. Then for all $\alpha\in \Delta_+$ and all $m\geq 0$, we have $[R(a^*(z))b_\alpha(z)]_m.\bw_n=0$. 
\end{lem} 

\begin{proof} Let us consider the case that $R(a^*(z))=a_\gamma^*(z)$. Then, we have 
\[
[a_\gamma^*(z)b_i(z)]_m.\bw_n= \sum_{m} \sum_{r+s=m} (a_{\gamma,r}^*.\bu_n)\otimes(b_{\alpha,s}.\bv_n). 
\]
Since $m\geq 0$, we either have that $s\geq n$ or $r\geq 1-n$. In the first case, $b_{\alpha,s}.\bv_n=0$ and in the second case $a_{\gamma,r}^*.\bu_n=0$. It follows that the above expression equals zero for all $m \geq 0$. The result for more general polynomials $R$ follows in a similar manner.  
\end{proof}

\sssec{} We claim that $f_{\alpha,m}.\bw_n=0$ for all $m\geq 0$.
\label{sss:basic}
Note that for all $m\in \bZ$, we have 
\begin{multline} 
f_{\alpha,m}.\bw_n=  \sum_{\beta\in \Delta_+} [:Q_\beta^\alpha(a^*(z))a_\beta(z):]_m.\bw_n + 
\sum_{\beta\in \Delta_+} [\tilde{Q}_\beta^\alpha(a^*(z)) \partial_z a_\beta^*(z)]_m.\bw_n +\\
 \sum_{\beta\in \Delta_+} [\tilde{Q}_\beta^\alpha(a^*(z)) \partial_z a_\beta^*(z)]_m.\bw_n + 
 [b_\alpha(z) a_\alpha^*(z)]_m.\bw_n + \sum_i [b_i(z) R_i^\alpha(a^*(z))]_m.\bw_n
\end{multline}
Now suppose $m\geq 0$.  By Corollary \ref{c:basic}.(ii), The first two terms vanishes. By Lemma \ref{l:basic2}, the last term vanishes. It is easy to check that the third term also vanishes.

\ssec{Proof of Proposition \ref{p:morphism}} 
Our goal is to construct a nontrivial morphism $\bM_\kappa (\Lambda) \ra \bW_\kappa(\Lambda)$. 
\sssec{}
Recall that, by definition, $\bW_\kappa(\Lambda)$ is equal to $M_{\fg,n}\otimes \pi_0(\Lambda')$ equipped with the $w_0$-twisted action of $\hfg_\kappa$, where $\Lambda'=w_0(\Lambda)-2\rho_n$. To construct our morphism, it is enough to show that 
\[
\fI_n\simeq t^n\fgbb\oplus \fb_n = t^n \fn\llb t \rrb \oplus t^n \fhbb \oplus t^n\fn^-\llb t \rrb \oplus \fb_n
\] 
acts on $\bw_n$ by the character $\Lambda$.

\sssec{} Recall that $-w_0$ defines a permutation of $\Delta_+$; in particular, twisting by $w_0$ sends $\fn$ to $\fn^-$.  
Now lemma \ref{l:basic}.(i) immediately implies that $t^n\fn^-\llb t\rrb\oplus t^n\fhbb$ annihilates $\bw_n$. Moreover, the discussion of \S \ref{sss:basic} shows  that $\fn_n$ also annihilates $\bw_n$.

\sssec{} It remains to compute the action of $\fh_n$ on $\bw_n$. (The computation that follows is very similar to the one performed in \cite[\S 9.5.1]{FrenkelBook}.) Write $\Lambda'=(\lambda_0', \cdots, \lambda_{n-1}')$. Let $m\in \{0,1,\cdots, n-1\}$. We have  
\[
 h_{i,m} .\bw_n = 
\left( \langle w_0(\lambda_m'), h_i\rangle -  \sum_{\beta\in \Delta_+} \langle w_0(\beta), h_{i} \rangle [a_\beta^*(z)a_\beta(z)]_m\right).\bw_n
\]
By Lemma \ref{l:basic0}, the above expression equals $\langle w_0(\lambda_m'), h_i\rangle.\bw_n$, unless $m=0$, in which case, it equals 
\[
\left(\langle  w_0(\lambda_0'), h_i \rangle - n \sum_{\beta \in \Delta_+} \langle w_0(\beta), h_i \rangle \right).\bw_n
= \langle w_0(\lambda_0')-2n. \rho, h_i \rangle. \bw_n
\]
We conclude that $\fh_n$ acts on $\bw_n$ via the character 
\[
w_0(\Lambda')-2\rho_n=w_0(w_0(\Lambda)-2\rho_n) -2\rho_n = \Lambda,
\] as required.  
\qed


\bibliographystyle{alpha}

\bibliography{ref.Center}

\end{document}